\documentclass{article}

\usepackage[english]{babel}

\usepackage[letterpaper,top=2cm,bottom=2cm,left=3cm,right=3cm,marginparwidth=1.75cm]{geometry}

\usepackage{amsmath}
\usepackage{amssymb}
\usepackage{amsthm}
\usepackage{graphicx}
\usepackage{geometry}
\usepackage{longtable}
\usepackage{array}
\usepackage[normalem]{ulem}
\usepackage{cancel}

\usepackage[
	textsize=tiny,
	textwidth=3.2cm,
	colorinlistoftodos,
	backgroundcolor=teal!30!white,
	linecolor=teal!50!white
	]{todonotes}

\makeatletter
\let\save@mathaccent\mathaccent
\newcommand*\if@single[3]{%
  \setbox0\hbox{${\mathaccent"0362{#1}}^H$}%
  \setbox2\hbox{${\mathaccent"0362{\kern0pt#1}}^H$}%
  \ifdim\ht0=\ht2 #3\else #2\fi
  }
\newcommand*\rel@kern[1]{\kern#1\dimexpr\macc@kerna}
\newcommand*\widebar[1]{\@ifnextchar^{{\wide@bar{#1}{0}}}{\wide@bar{#1}{1}}}
\newcommand*\wide@bar[2]{\if@single{#1}{\wide@bar@{#1}{#2}{1}}{\wide@bar@{#1}{#2}{2}}}
\newcommand*\wide@bar@[3]{%
  \begingroup
  \def\mathaccent##1##2{%
    \let\mathaccent\save@mathaccent
    \if#32 \let\macc@nucleus\first@char \fi
    \setbox\z@\hbox{$\macc@style{\macc@nucleus}_{}$}%
    \setbox\tw@\hbox{$\macc@style{\macc@nucleus}{}_{}$}%
    \dimen@\wd\tw@
    \advance\dimen@-\wd\z@
    \divide\dimen@ 3
    \@tempdima\wd\tw@
    \advance\@tempdima-\scriptspace
    \divide\@tempdima 10
    \advance\dimen@-\@tempdima
    \ifdim\dimen@>\z@ \dimen@0pt\fi
    \rel@kern{0.6}\kern-\dimen@
    \if#31
      \overline{\rel@kern{-0.6}\kern\dimen@\macc@nucleus\rel@kern{0.4}\kern\dimen@}%
      \advance\dimen@0.4\dimexpr\macc@kerna
      \let\final@kern#2%
      \ifdim\dimen@<\z@ \let\final@kern1\fi
      \if\final@kern1 \kern-\dimen@\fi
    \else
      \overline{\rel@kern{-0.6}\kern\dimen@#1}%
    \fi
  }%
  \macc@depth\@ne
  \let\math@bgroup\@empty \let\math@egroup\macc@set@skewchar
  \mathsurround\z@ \frozen@everymath{\mathgroup\macc@group\relax}%
  \macc@set@skewchar\relax
  \let\mathaccentV\macc@nested@a
  \if#31
    \macc@nested@a\relax111{#1}%
  \else
    \def\gobble@till@marker##1\endmarker{}%
    \futurelet\first@char\gobble@till@marker#1\endmarker
    \ifcat\noexpand\first@char A\else
      \def\first@char{}%
    \fi
    \macc@nested@a\relax111{\first@char}%
  \fi
  \endgroup
}
\makeatother

\newcommand{\keywords}[1]{\def\mykeywords{#1}}
\newcommand{\subjclass}[2][2020]{\def\mysubjclassLabel{#1 MSC:}\def\mysubjclass{#2}}

\let\oldabstract\abstract
\let\endoldabstract\endabstract
\renewenvironment{abstract}{%
  \oldabstract%
}{%
  \par\vspace{0.5cm}
  \ifx\mykeywords\undefined\else
    \noindent\textbf{Keywords:} \mykeywords\par
  \fi
  \ifx\mysubjclass\undefined\else
    \vspace{0.2cm}
    \noindent\textbf{\mysubjclassLabel} \mysubjclass\par
  \fi
  \endoldabstract%
}

\usepackage[colorlinks=true, allcolors=blue]{hyperref}
\usepackage[abbrev]{amsrefs}
\allowdisplaybreaks

\newcommand{\N}{\mathbb{N}}
\newcommand{\R}{\mathbb{R}}
\newcommand{\rd}{\R^d}
\newcommand{\D}{\mathcal{D}}
\newcommand{\Cinf}{\mathcal{C}^\infty}
\newcommand{\Lpd}{L^p_{d-1}}
\newcommand{\cVec}[1]{\mathbf{e}_{#1}}
\newcommand{\ej}{\cVec{j}}
\newtheorem{theorem}{Theorem}
\newtheorem{lemma}{Lemma}

\title{Higher-order derivatives of radially symmetric functions}
\author{Zdeněk Mihula\footnote{Department of Mathematics, Faculty of Electrical Engineering, Czech Technical University in Prague, Technická 2, 166 27 Praha 6, Czech Republic, Email address: \href{mihulzde@fel.cvut.cz}{mihulzde@fel.cvut.cz}} and
Jan Vybíral\footnote{Department of Mathematics, Faculty of Nuclear Sciences and Physical Engineering, Czech Technical University in Prague, Trojanova 13, 12000 Praha, Czech Republic, Email address: \href{jan.vybiral@fjfi.cvut.cz}{jan.vybiral@fjfi.cvut.cz}, J.V. is a member of the Nečas center for mathematical modeling.}}



\begin{document}
\maketitle

\begin{abstract}
We prove a surprisingly simple pointwise formula for the Frobenius norm of the tensor of $n$-th order partial derivatives of
a radially symmetric function $f(x)=g(r(x))$. Using the iterations of the differential operator $\D g(r)=g'(r)/r$,
we avoid technical difficulties usually caused by higher-order radial derivatives. As a consequence, we obtain a complete characterization
of the subspace of radially symmetric functions in both inhomogeneous and homogeneous Sobolev spaces of arbitrarily high order
and all integrability parameters $p\in [1,\infty).$ Furthermore, the pointwise nature of our approach allows us to obtain similar results
also for Sobolev-type spaces built upon more general Banach lattices.
\end{abstract}

\section{Introduction} 

Radially symmetric functions (or \emph{radial functions} for short) are functions of the form $f(x)=g(r(x))$, where $g$ is a function of
a single variable and $r(x)=\sqrt{x_1^2+\dots+x_d^2}$ is the Euclidean distance from the origin in $\rd$ for $d\ge 2$.
They play a prominent role in mathematical analysis and mathematical physics since the pioneering work of Newton,
Laplace, and Fourier. They appear in Schwarzschild's solution of the Einstein field equations 
as well as in the quantum mechanical solution of Schrödinger's equation for a hydrogen atom.

Radially symmetric functions are elegant objects whose symmetry often brings in pleasant properties that general functions of several variables
do not possess. They often form an important subspace of various function spaces measuring smoothness and integrability of functions, such
as Sobolev spaces $W^{n,p}(B_R)$, where by $B_R$ we denote the open ball in $\rd$ centered at the origin with radius $R\in(0,\infty]$.
When $R=\infty$, we interpret $B_R$ as $\rd$. A systematic study of radial functions in Sobolev spaces was initiated after
the seminal paper \cite{MR0454365} of Strauss.

In his paper, Strauss proved the existence of so-called solitary waves in the nonlinear Klein--Gordon equation in higher dimensions. A fundamental obstacle that Strauss faced was that unlike on bounded domains, the Sobolev space $W^{1,2}(\rd)=H^1(\rd)$ does not compactly embed into any Lebesgue space $L^q(\rd)$. Depending on the value of $q$, there are different phenomena causing the loss of compactness (see~\cites{Lions-vanA,Lions-vanB,Lions-conA,Lions-conB} for more information), but one is ever present\textemdash the translation invariance, which allows the mass to escape to infinity. Nevertheless, Strauss' key realization was that since he did not need to work with the entire $H^1(\rd)$ but only with its subspace $H_{rad}^1(\rd)$ consisting of radially symmetric functions, the symmetry prevents the mass from escaping. He established that a radial function $f\in H_{rad}^1(\rd)$ is necessarily continuous outside the origin and satisfies the decay estimate
\begin{equation*}
    |f(x)| \leq C_d r(x)^{\frac{1-d}{2}} \|f\|_{H^1(\rd)} \quad \text{for every $x\neq0$},
\end{equation*}
where $C_d$ is a constant depending only on the dimension $d\geq2$.

This estimate, which is now usually called \emph{Strauss' radial lemma}, was subsequently a key ingredient in proving that the restricted Sobolev space $H_{rad}^1(\rd)$ embeds compactly into $L^q(\rd)$ for suitable values of $q$. Later, the role of symmetry on compactness and decay of functions in Sobolev spaces was comprehensively studied by Lions in \cite{Lions-com}. Strauss' work was accompanied by the parallel work \cite{C-G-M} of Coleman, Glaser, and Martin, who showed that ground state solutions of many Euclidean scalar field equations are radially symmetric, cementing the subspace of radial functions in Sobolev spaces as a natural important subspace to study in detail. Later, Berestycki and Lions developed a unifying theory for this in \cites{B-L1,B-L2,B-L3}. As a final aside, the recovery of compactness can seem surprising considering that radial functions are, from a different point of view, the worst possible in Sobolev embeddings in view of the classical symmetrization principles (see~\cites{Baernstein,Talenti}). In fact, there is an elegant geometric reason behind this (see~\cite{MR4277332} and references therein).
Ever since these foundational works were published, the subspace of radial functions in various Sobolev spaces has been
intensively studied \cite{GdF-dS-M,Win1,Win2,MR3330617}.

To describe the results about subspaces of radial functions, we first recall some notation.
Let $f\colon\Omega\to \R$ be a (sufficiently smooth) function defined on an open set $\Omega\subset\R^d$.
We denote by $\nabla^n f(x)$ the \emph{tensor of all $n$-th order partial derivatives} of $f$, that is,
\[
\nabla^n f(x)=(\nabla^n f(x))_{i_1,\dots,i_n=1}^d,\quad \text{where}\quad (\nabla^n f(x))_{i_1,\dots,i_n}=\frac{\partial^n f(x)}{\partial x_{i_1}\dots\partial x_{i_n}},
\]
and by $\|\nabla^n f(x)\|_F$ its norm
\begin{equation*}
    \|\nabla^n f(x)\|_F = \sqrt{\sum_{i_1,\dots,i_n=1}^d \Bigl(\frac{\partial^nf(x)}{\partial x_{i_1}\dots\partial x_{i_n}}\Bigr)^2}.
\end{equation*}

Note that $\|\nabla^n f(x)\|_F$ is the \emph{Frobenius norm} of $\nabla^n f(x)$.
For $p\in[1, \infty)$, the (inhomogeneous) \emph{Sobolev space} $W^{n,p}(B_R)$ is defined as the completion of smooth functions on the closed ball $\widebar{B_R}$, that is, of $\Cinf(\widebar{B_R})$-functions, with respect to the norm
 \begin{equation}\label{E:inhom_Sob}
     \|f\|_{W^{n,p}(B_R)} = \sum_{k = 0}^n \|\nabla^k f\|_{L^{p}(B_R)}.
 \end{equation}
 For brevity, we write $\|\nabla^k f\|_{L^{p}(B_R)}$ instead of $\|\|\nabla^k f\|_F\|_{L^{p}(B_R)}$. When $R=\infty$, we will equivalently consider the completion of $\Cinf(\rd)$-functions with compact support. 

A classical problem from the theory of radial functions is to characterize when a radial function $f(x) = g(r(x))$ belongs to the
Sobolev space $W^{n,p}(B_R)$ by means of its \emph{radial profile} $g$. For a locally integrable radial function $f$, its radial profile $g$ is defined as $g(t)=f(t\cdot\cVec{1})$, which is a well-defined measurable function on $(-R,R)$ (see~\cite[p.~5]{SSV1}). Note that if $f(x) = g(r(x))\in\Cinf(\widebar{B_R})$, then its radial profile $g$ belongs to $\Cinf_{even}([-R,R])$. By $\Cinf_{even}([-R,R])$, we denote the space of \emph{even $\Cinf([-R,R])$-functions}. When $R=\infty$, we interpret the interval $[-R,R]$ as $\R$. On the other hand, for every $g\in\Cinf_{even}([-R,R])$, the function $f$ defined as $f(x) = g(r(x))$ is a smooth radial function (cf.~\cite{Whitney:43}).

In \cite{SSV1}, such a characterization was provided for the even-order Sobolev space $W^{2n,p}(\rd)$, $p\in(1 ,\infty)$, by means of iterations of the \emph{radial Laplacian}
\begin{equation}\label{E:radial_laplacian}
    \Delta_{r} f(x) = D_r g(r) = g''(r) + (d-1)\,\frac{g'(r)}{r} \Big\rvert_{r=r(x)}
\end{equation}
and a suitable \emph{weighted Lebesgue space}. For $-\infty\leq a < b\leq\infty$, we will denote  by $\Lpd(a,b)$ the weighted Lebesgue space with the weight $|t|^{d-1}$, that is,
\begin{equation*}
    \|h\|_{\Lpd(a,b)}^p = \int_a^b |h(t)|^p |t|^{d-1} \,dt
\end{equation*}
for every measurable function $h$ on $(a, b)$. They showed that a radial function $f(x) = g(r(x))$ belongs to $W^{2n,p}(\rd)$ if and only if its radial profile $g$ belongs to the closure of compactly supported $\Cinf_{even}(\R)$-functions with respect to the norm
\begin{equation*}
\|g\|_{\Lpd(\R)} + \|D_r^n g\|_{\Lpd(\R)}.
\end{equation*}
However, not only is this characterization inherently limited to even-order Sobolev spaces, but it does not extend to the endpoints $p\in\{1,\infty\}$ either. This is rooted in the fact that Riesz transforms are not bounded on $L^1(\rd)$ and $L^\infty(\rd)$ (see~\cite{Stein}).

Around the same time, a different characterization was obtained in \cite{GdF-dS-M} for the integer-order Sobolev space $W^{n,p}(B_R)$ on the open ball $B_R$ with $R\in(0, \infty)$. While their characterization works for any integer order and also $p=1$, the parameters $n\in\N$ and $p\geq1$ have to satisfy $(n-1)p < d$, which considerably restricts one when the other is fixed. Under this restriction, it follows from their results that a radial function $f(x) = g(r(x))$ belongs to $W^{n,p}(B_R)$ if and only if its radial profile $g$ belongs to the closure of $\Cinf_{even}([-R,R])$ with respect to the norm
\begin{equation*}
\sum_{k=0}^n \|g^{(k)}\|_{\Lpd(-R,R)}.
\end{equation*}

Recently, a characterization without any restrictions on $n\in\N$ and $p\in[1, \infty)$ was obtained in \cite{Ostermann}. It draws from the short beautiful paper \cite{Lyons-Zumbrun}, which contains an elegant formula for partial derivatives of radial functions. Since we think \cite{Lyons-Zumbrun} deserves more attention than it has obtained so far, we recall the formula here. Given a multi-index $\alpha=(\alpha_1,\dots,\alpha_d)\in\N_0^d$ with $|\alpha|=\alpha_1+\dots+\alpha_d = n$ and a smooth radial function $f(x) = g(r(x))$, the formula asserts that
\begin{equation}\label{E:Lyons-Zumbrun}
    \frac{\partial^n f(x)}{\partial x_1^{\alpha_1}\dots\partial x_d^{\alpha_d}} = \sum_{k = 0}^{\lfloor n/2 \rfloor} \frac1{2^k k!} \Delta^k(x^\alpha) \cdot \D^{n - k} g(r),
\end{equation}
where $x^\alpha = x_1^{\alpha_1}\dots x_d^{\alpha_d}$, $\Delta$ is the Laplace operator, and $\D g(r)=g'(r)/r$. With the help of \eqref{E:Lyons-Zumbrun} and Hardy-type inequalities, a characterization in terms of the modified radial profile $f(x)=h(r(x)^2)$ was proved in \cite{Ostermann}. With the usual radial profile, it can be formulated as follows. A radial function $f(x) = g(r(x))$ belongs to $W^{n,p}(B_R)$ if and only if $g$ belongs to the closure of $\Cinf_{even}([-R,R])$ with respect to the norm
\begin{equation*}
\sum_{k=0}^n \|t^k \D^k g(t)\|_{\Lpd(-R,R)}.
\end{equation*}

\bigskip

The main object of our study is the square of the Frobenius norm of the tensor of $n$-th order partial derivatives of $f$
\begin{equation}\label{eq:defS_n}
(S_nf)(x)=\|\nabla^n f(x)\|_F^2=\sum_{i_1,\dots,i_n=1}^d \Bigl(\frac{\partial^nf(x)}{\partial x_{i_1}\dots\partial x_{i_n}}\Bigr)^2.
\end{equation}
Note that this expression is by no means new. It is frequently used for $n=2$ but various problems require working with $(S_n f)(x)$
also for larger values of $n\geq3$. For example, it appears in the calculus of variations as the so-called \emph{$n$-harmonic energy functional}
or \emph{$n$-polyenergy} (see~\cite{AP,GS}). Furthermore, $(S_n f)(x)$ for large values of $n$ is used in the construction
of \emph{polyharmonic splines}
(see \cite[Chapter 2.4]{W} or \cite{Duchon}).

Note that if one had a suitably simple formula for $(S_n f)(x)$ for a radial function $f$ in terms of its radial profile $g$, it would directly lead to a simple characterization of the subspace of radial functions (see~Section~\ref{sec:subspace} for more information). Of course, generalizations of the chain rule for higher-order derivatives are available, such as Faà di Bruno's formula (see \cite{Johnson:02} and references therein), but they usually do not lead to appealing formulas.  Therefore, it would be desirable to have a simple, elegant closed-form formula for $(S_n f)(x)$. 
However, such a formula for general $n\in\N$ is not available in the literature, to the best of our knowledge. The main goal of this paper is to close this gap.

For small values of $n$ it is quite easy to observe (cf.~\cite[Theorem 6]{SSV1}) that $(S_nf)(x)$ takes a particularly elegant form if $f(x)=g(r(x))$ is a radially symmetric function. Indeed, a direct calculation shows that
\begin{equation}\label{eq:Sn1}
(S_1f)(x)=\sum_{j=1}^d \Bigl(\frac{\partial f(x)}{\partial x_j}\Bigr)^2=\sum_{j=1}^d \Bigl(\frac{g'(r)}{r}\cdot x_j\Bigr)^2=[g'(r)]^2,
\end{equation}
where we write $r$ instead of $r(x)$ to simplify the notation. A similar formula for $n=2$ can still be obtained in a rather straightforward way; one arrives at
\begin{equation}\label{eq:Sn2}
(S_2 f)(x)=[g''(r)]^2 + \frac{d-1}{r^2}[g'(r)]^2.
\end{equation}
This result is well-known and has a classical geometric interpretation. The spectrum of the Hessian $Hf(x)=\nabla^2f(x)$
is rotationally invariant and consists of one eigenvalue $g''(r)$ and the eigenvalue $g'(r)/r$ with multiplicity $d-1$.
The square of the Frobenius norm of $Hf(x)$ is then the sum of the squares of these eigenvalues. Note that the radial Laplacian \eqref{E:radial_laplacian} is the sum of these eigenvalues.

For $n\ge 3$, the calculation of $(S_nf)(x)$  becomes quickly very technical and time-consuming. Nevertheless, one can still directly verify that
\[
(S_3f)(x)=(g'''(r))^2 + 3(d-1) \left( \frac{g''(r)}{r} - \frac{g'(r)}{r^2} \right)^2,
\]
which hints that an appealing general formula for $S_n f$ might exist. This was conjectured during the work on \cite{SSV1},
but the problem remained unsolved until now. The following theorem finally solves this problem. Recall that we denote $\D g(r)=g'(r)/r$.
\begin{theorem}\label{thm:main'}
Let $n\ge 1$, $d\geq2$, and let $f(x)=g(r(x))$ be a smooth radial function. Then $(S_n f)(x)$
is a radial function with
\begin{equation*}
    (S_nf)(x)=(T_ng)(r(x)),
\end{equation*}
where
\begin{align}\label{eq:defT'}
(T_ng)(r)&=\sum_{j=0}^{\lfloor n/2\rfloor} \alpha^d_{n,j} \left[(\D^j g)^{(n-2j)}(r)\right]^2,
\end{align}
where $\alpha_{n,j}^d$ are positive integers with $\alpha^d_{n,0}=1$ and
\[
\alpha^d_{n,j}= \binom{n}{2j}\cdot (2j-1)!!\cdot \prod_{k=0}^{j-1} (d-1+2k),\quad 1\le j\le \lfloor n/2\rfloor.
\]
\end{theorem}
It is worth pointing out that the formula \eqref{eq:defT'} is pointwise and does not involve any particular function norms.
The structure of the paper is as follows. In Section~\ref{sec:2} we give the proof of Theorem~\ref{thm:main'}.
In Section~\ref{sec:subspace} we derive the characterizations of subspaces of radially symmetric functions of homogeneous 
and inhomogeneous Sobolev spaces and discuss their relation to the results of \cite{Ostermann}.
Finally, in Section~\ref{sec:4} we complement our findings with a discussion of the radial symmetry of $S_nf$ and
alternative formulas for $T_n$ followed by a characterization of subspaces of radially symmetric functions of more general Sobolev spaces
and by the list of open problems.

\section{Proof of Theorem~\ref{thm:main'}}\label{sec:2}
Our proof of Theorem~\ref{thm:main'} will be based on induction. As the first step, we state the following simple fact.

\begin{lemma}\label{lem:1} Let $F$ be a smooth function defined on a neighborhood of $x\in\R^d.$ Then
\begin{align*}
(S_{n+1}F)(x)&=\sum_{i_1,\dots,i_n=1}^d\sum_{i_{n+1}=1}^d \Bigl(\frac{\partial^{n+1}F(x)}{\partial x_{i_1}\dots\partial x_{i_{n+1}}}\Bigr)^2
=\sum_{i_1,\dots,i_n=1}^d\sum_{j=1}^d \Bigl(\frac{\partial^{n}}{\partial x_{i_1}\dots\partial x_{i_{n}}}\frac{\partial F(x)}{\partial x_j}\Bigr)^2 \nonumber\\
&=\sum_{j=1}^d S_n\Bigl(\frac{\partial F}{\partial x_j}\Bigr)(x). 
\end{align*}
\end{lemma}

The next lemma studies the quantity $S_n (F)$ for functions $F(x)=G(x)x_j$ for $j=1,\dots,d.$ Such functions appear naturally as gradients of radially symmetric functions
and will play a crucial role later on.

\begin{lemma}\label{lem:2} Let $G$ be a smooth function defined on a neighborhood of $x\in \R^d$. Then
\begin{align*}
    \sum_{j=1}^d (S_n(G(x)x_j))(x)=r(x)^2\cdot (S_n G)(x)+n\sum_{j=1}^d x_j\frac{\partial}{\partial x_j} (S_{n-1}G)(x)
    +n(d+n-1) (S_{n-1}G)(x).
\end{align*}
\end{lemma}
\begin{proof}
It will be more convenient for us to express $S_n f$ with the help of multi-indices. Namely, for a multi-index $\alpha=(\alpha_1,\dots,\alpha_d)\in\N_0^d$,
we denote
\[
D^\alpha f(x)=\frac{\partial^{|\alpha|}f(x)}{\partial x_1^{\alpha_1}\dots\partial x_d^{\alpha_d}},
\]
where $|\alpha|=\alpha_1+\dots+\alpha_d$.
Using the symmetry of partial derivatives we can rewrite \eqref{eq:defS_n} as
\begin{equation}\label{eq:lem2_1}
    (S_nf)(x)=\sum_{|\alpha|=n} \frac{n!}{\alpha!}(D^\alpha f)(x)^2.
\end{equation}
As usual, $\alpha!=\alpha_1!\cdot\ldots\cdot \alpha_d!$ is the factorial of the multi-index $\alpha.$
We apply \eqref{eq:lem2_1} to $G(x)x_j$ and sum up the result over $j=1,\dots,d$ and obtain
\begin{equation}\label{eq:SnGx}
\sum_{j=1}^d (S_n(G(x)x_j))(x)=\sum_{j=1}^d \sum_{|\alpha|=n} \frac{n!}{\alpha!}(D^\alpha[G(x)\cdot x_j])(x)^2.
\end{equation}
A direct application of the Leibniz product rule gives
\begin{equation}\label{eq:difalphaj}
D^\alpha[G(x)\cdot x_j](x)=(D^\alpha G)(x)\cdot x_j+\alpha_j (D^{\alpha-\ej}G)(x),
\end{equation}
where $\ej$ is the $j$-th canonical basis vector and the second term is interpreted as zero if $\alpha_j=0$. We take the square of \eqref{eq:difalphaj}, which yields
\[
D^\alpha[G(x)\cdot x_j](x)^2=(D^\alpha G)(x)^2\cdot x^2_j+2\alpha_jx_j (D^{\alpha}G)(x)(D^{\alpha-\ej}G)(x)+
\alpha^2_j (D^{\alpha-\ej}G)(x)^2.
\]
We plug this into \eqref{eq:SnGx} and obtain
\begin{align*}
 \sum_{j=1}^d (S_n(G(x)x_j))(x)&=I+II+III
 \end{align*}
 where
 \begin{align*}
 I&=\sum_{j=1}^d \sum_{|\alpha|=n} \frac{n!}{\alpha!} (D^\alpha G)(x)^2\cdot x^2_j=\sum_{j=1}^d x_j^2\cdot \sum_{|\alpha|=n} \frac{n!}{\alpha!} (D^\alpha G)(x)^2 = r(x)^2\cdot (S_n G)(x),\\
 II&=2\sum_{j=1}^d \sum_{|\alpha|=n} \frac{n!}{\alpha!} \alpha_jx_j (D^{\alpha}G)(x)(D^{\alpha-\ej}G)(x),\\
 III&=\sum_{j=1}^d \sum_{|\alpha|=n} \frac{n!}{\alpha!} \alpha^2_j (D^{\alpha-\ej}G)(x)^2.
\end{align*}
To simplify the third term, note that for every $j=1,\dots,d$ and every $\alpha\in \N_0^d$ with $|\alpha|=n$ and $\alpha_j\ge 1$, the multi-index $\beta=\alpha-\ej\in\N_0^d$
is well-defined and satisfies $|\beta|=n-1.$ Therefore, we can rewrite
\begin{align*}
III
&=\sum_{|\beta|=n-1} \sum_{j=1}^d  \frac{n\cdot (n-1)!}{\beta!} (\beta_j+1) D^{\beta}G(x)^2\\
&=n\sum_{|\beta|=n-1}   \frac{(n-1)!}{\beta!}  D^{\beta}G(x)^2\cdot \sum_{j=1}^d (\beta_j+1)=n(d+n-1) (S_{n-1}G)(x).
\end{align*}
Finally, we rewrite the second term as
\begin{align*}
    II
    &=2\sum_{|\beta|=n-1}\sum_{j=1}^d \frac{n!}{(\beta+\ej)!}(\beta_j+1)x_j (D^{\beta+\ej}G)(x)(D^\beta G)(x)\\
    &=2n\sum_{|\beta|=n-1}\frac{(n-1)!}{\beta!}(D^\beta G)(x)\sum_{j=1}^d x_j\cdot \frac{\partial}{\partial x_j}(D^{\beta}G)(x)\\
    &=n\sum_{|\beta|=n-1}\frac{(n-1)!}{\beta!} \sum_{j=1}^d x_j\cdot \frac{\partial}{\partial x_j} (D^\beta G)(x)^2=n\sum_{j=1}^d x_j\frac{\partial}{\partial x_j} (S_{n-1}G)(x).
\end{align*}

We plug all these results into \eqref{eq:SnGx} and obtain
\begin{align*}
    \sum_{j=1}^d (S_n(G(x)x_j))(x)=r(x)^2\cdot (S_n G)(x)+n\sum_{j=1}^d x_j\frac{\partial}{\partial x_j} (S_{n-1}G)(x)
    +n(d+n-1) (S_{n-1}G)(x),
\end{align*}
which is the desired formula.
\end{proof}
Finally, we are in a position to prove Theorem~\ref{thm:main'}.
\begin{proof}[Proof of Theorem~\ref{thm:main'}]
Let $f(x)=g(r(x))$ be a smooth radial function. Then
\[
\frac{\partial f(x)}{\partial x_j}=\frac{g'(r(x))}{r(x)}\cdot x_j= \D g(r(x))\cdot x_j.
\]
We now apply Lemma~\ref{lem:1} and obtain
\begin{align*}
(S_{n+1}f)(x)&=\sum_{j=1}^d S_n\Bigl(\frac{\partial f}{\partial x_j}\Bigr)(x)
=\sum_{j=1}^d S_n\bigl(\D g(r)\cdot x_j\bigr)(x).
\end{align*}
We now assume that the formula $(S_\ell h)(x)=(T_\ell \psi)(r(x))$ is true for every $1\le \ell \le n$ and every smooth radial function $h(x)=\psi(r(x))$.
Then we use this fact for $h(x)=\D g(r(x))$ and $\psi=\D g$, combine it with Lemma~\ref{lem:2}
and the identity $x_j\cdot\frac{\partial}{\partial x_j}=x_j^2 {\mathcal D}=\frac{x_j^2}{r}\cdot\frac{d}{dr}$,  and arrive at
\begin{align}
\notag(S_{n+1}f)(x)&= r^2\cdot T_n (\D g)(r)+n\sum_{j=1}^d \frac{x_j^2}{r}\cdot\frac{d}{dr} (T_{n-1}\D g)(r)
    +n(d+n-1) (T_{n-1}\D g)(r)\\
\label{eq:induction}&=r^2\cdot  T_n (\D g)(r)+nr\frac{d}{d r}\left[T_{n-1} \D g\right](r)+n(d+n-1)(T_{n-1}\D g)(r).
\end{align}

The proof of Theorem \ref{thm:main'} will be finished once we show that the right hand side of \eqref{eq:induction}
can be rewritten as the right hand side of \eqref{eq:defT'} with $n$ replaced by $n+1$.

\medskip

For future reference, we start proving this by listing two identities concerning the coefficients $\alpha^d_{n,j}$, which can be easily verified:
\begin{align}
    \alpha^d_{n,j} &= \frac{n}{n-2j}\alpha^d_{n-1,j} \quad \text{for every $j = 0, \dots, \Big\lfloor \frac{n-1}{2} \Big\rfloor$} \label{E:coef_alpha_identity1}
    \intertext{and}
    \frac{2j}{n-2j+1}\alpha^d_{n,j} &= (n-2j+2)(d+2j-3)\alpha^d_{n,j-1} \quad \text{for every $j = 1, \dots, \Big\lfloor \frac{n}{2} \Big\rfloor$}. \label{E:coef_alpha_identity2}
\end{align}

We proceed by induction based on \eqref{eq:induction}. We start by observing that \eqref{eq:defT'} holds for $n=1$ and $n=2$.
Indeed, using that $\alpha^d_{1,0} = \alpha^d_{2,0} = 1$ and $\alpha^d_{2,1} = d-1$ we recover \eqref{eq:Sn1} and \eqref{eq:Sn2}.
Let us assume that \eqref{eq:defT'} holds up to $n\in\N$. We want to show that
\begin{equation}\label{eq:ind1}
(T_{n+1}g)(r)=r^2 T_n(\D g)(r)+nr\frac{d}{d r}\left[T_{n-1}\D g\right](r)+n(d+n-1)(T_{n-1}\D g)(r)
\end{equation}
takes also the form of \eqref{eq:defT'}.

First, we observe that $\varphi'(r)=r\D\varphi(r)$ for every smooth function $\varphi$.
We differentiate this identity and obtain $\varphi''(r)=r(\D \varphi)'(r)+\D\varphi (r)$, $\varphi'''(r)=r(\D \varphi)''(r)+2(\D\varphi)' (r)$
and, in general,
\begin{equation}\label{eq:diff_relation}
    \varphi^{(k+1)}(r)=r(\D \varphi)^{(k)}(r)+k(\D\varphi)^{(k-1)} (r) \quad \text{for every $k\in\N_0$},
\end{equation}
in which the second term is to be interpreted as $0$ when $k=0$. We apply this identity to $\varphi=\D^jg$ with $k=n-2j$ and after taking squares and reordering the terms it becomes
\begin{align*}
[r(\D^{j+1}g)^{(n-2j)}]^2 &+ 2r(n-2j)(\D^{j+1}g)^{(n-2j)}(\D^{j+1} g)^{(n-2j-1)} \\
&=[(\D^j g)^{(n+1-2j)}]^2-(n-2j)^2[(\D^{j+1} g)^{(n-2j-1)}]^2.
\end{align*}
Now we multiply this identity with $\alpha^d_{n,j}$, use the identity \eqref{E:coef_alpha_identity1}, and sum up over $j=0,\dots,\lfloor (n-1)/2\rfloor$ to obtain

\begin{align}
\label{eq:proof_1}\sum_{j=0}^{\lfloor (n-1)/2\rfloor} &\alpha^d_{n,j}\, \left[r(\D^{j+1} g)^{(n-2j)}\right]^2
+ 2nr\sum_{j=0}^{\lfloor (n-1)/2\rfloor} \alpha^d_{n-1,j} (\D^{j+1} g)^{(n-1-2j)} (\D^{j+1} g)^{(n-2j)}\\
\notag&=\sum_{j=0}^{\lfloor (n-1)/2\rfloor}\alpha^d_{n,j}[(\D^j g)^{(n+1-2j)}]^2-\sum_{j=0}^{\lfloor (n-1)/2\rfloor}\alpha^d_{n,j}(n-2j)^2[(\D^{j+1} g)^{(n-2j-1)}]^2.
\end{align}
Observe that the left-hand side of \eqref{eq:proof_1} is equal to
\[
r^2 T_n(\D g)(r)+nr\frac{d}{d r}\left[T_{n-1}(\D g)\right](r)
\]
if $n$ is odd. If $n$ is even, that same is true if we add the term corresponding to $j=n/2$, namely $\alpha_{n,n/2}^d[r\D^{n/2+1}g]^2$.

Next, we add 
\[
n(d+n-1)(T_{n-1}\D g)(r)=n(d+n-1) \sum_{j=0}^{\lfloor (n-1)/2\rfloor} \alpha^d_{n-1,j} \left[(\D^{j+1} g)^{(n-1-2j)}\right]^2
\]
to the second term on the right-hand side of \eqref{eq:proof_1}, use the identity \eqref{E:coef_alpha_identity1}, and obtain
\begin{align*}
\sum_{j=0}^{\lfloor (n-1)/2\rfloor}&\Bigl[n(d+n-1) \alpha^d_{n-1,j}-\alpha^d_{n,j}(n-2j)^2\Bigr]\cdot[(\D^{j+1} g)^{(n-2j-1)}]^2\\
&=\sum_{j=0}^{\lfloor (n-1)/2\rfloor}\alpha^d_{n,j}(n-2j)(d+2j-1)[(\D^{j+1} g)^{(n-2j-1)}]^2.
\end{align*}

After an index shift from $j+1$ to $j$ and using the identity \eqref{E:coef_alpha_identity2}, this can be further simplified to
\[
\sum_{j=1}^{\lfloor (n+1)/2\rfloor}\frac{2j\cdot \alpha^d_{n,j}}{n-2j+1}[(\D^{j} g)^{(n-2j+1)}]^2
\]
if $n$ is even. Note that if $n$ is odd, then it is still true if we keep the last term with $j=(n+1)/2$ in its original form 
\[
\alpha^d_{n,(n-1)/2}(d+n-2)[\D^{(n+1)/2} g]^2.
\]

\bigskip

To summarize, the right-hand side of \eqref{eq:ind1} is equal to
\begin{equation}\label{eq:todo1a}
\sum_{j=0}^{\lfloor (n-1)/2\rfloor}\alpha^d_{n,j}[(\D^j g)^{(n+1-2j)}]^2+\sum_{j=1}^{\lfloor n/2\rfloor}\frac{2j\cdot \alpha^d_{n,j}}{n-2j+1}[(\D^{j} g)^{(n+1-2j)}]^2
\end{equation}
plus either
\begin{equation}\label{eq:todo1b_even}
    \alpha_{n,n/2}^d[r\D^{n/2+1}g]^2
    \quad \text{when $n$ is even}
\end{equation}
or
\begin{equation}\label{eq:todo1b_odd}
    \alpha^d_{n,(n-1)/2}(d+n-2)[\D^{(n+1)/2} g]^2 \quad \text{when $n$ is odd}.
\end{equation}

\bigskip

Recall, that we want to show that
\begin{equation}\label{eq:todo2}
(T_{n+1}g)(r)=\sum_{j=0}^{\lfloor (n+1)/2\rfloor} \alpha^d_{n+1,j} \left[(\D^j g)^{(n+1-2j)}\right]^2.
\end{equation}
We do it simply by comparing the decompositions \eqref{eq:todo2} and \eqref{eq:todo1a} plus either \eqref{eq:todo1b_even} or \eqref{eq:todo1b_odd} term-by-term. Indeed,
the coefficient of $[g^{(n+1)}]^2$ in \eqref{eq:todo1a} is $\alpha^d_{n,0}=1=\alpha^d_{n+1,0}$, in agreement with \eqref{eq:todo2}.
If $1\le j\le \lfloor (n-1)/2\rfloor$, then the coefficient of the factor $[(\D^j g)^{(n+1-2j)}]^2$ in \eqref{eq:todo1a} is
\[
\alpha^d_{n,j}+\frac{2j\, \alpha^d_{n,j}}{n-2j+1}=\alpha^d_{n,j}\cdot \frac{n+1}{n-2j+1}=\alpha^d_{n+1,j},
\]
which again corresponds to \eqref{eq:todo2}.

\bigskip

Finally, we deal with the remaining $\lfloor (n-1)/2\rfloor<j = \lfloor (n+1)/2\rfloor$. If $n$ is even, we are left with $j=n/2$ and we need to check that
\[
\alpha^d_{n+1,n/2}[(\D^{n/2}g)']^2=\alpha^d_{n,n/2}[r\D^{n/2+1}g]^2+n\alpha^d_{n,n/2}[(\D^{n/2}g)']^2,
\]
which follows from $\alpha^d_{n+1,n/2}=(n+1)\alpha^d_{n,n/2}$ and $\D^{n/2+1}g=(\D^{n/2}g)'/r$. At last, if $n$ is odd and $j=(n+1)/2$, we combine the identities \eqref{E:coef_alpha_identity2} and \eqref{E:coef_alpha_identity1} to obtain
\[
\alpha^d_{n+1,(n+1)/2}\bigl[\D^{(n+1)/2}g\bigr]^2=\alpha^d_{n,(n-1)/2}\cdot(d+n-2)\bigl[\D^{(n+1)/2}g\bigr]^2,
\]
which finishes the proof.

\end{proof}

\section{Subspaces of radial functions}\label{sec:subspace}
It is a classical problem to characterize when a radial function $f(x) = g(r(x))$ on the ball $B_R$, $R\in(0, \infty]$, belongs to a Sobolev space 
in terms of its radial profile $g$. We will soon see that such a characterization follows immediately from the pointwise formula
for $S_n f$ obtained in Theorem~\ref{thm:main'}.

Let $p\in[1, \infty)$ and $n\in\N$. Besides the inhomogeneous Sobolev spaces $W^{n,p}(B_R)$  (recall~\eqref{E:inhom_Sob}), their homogeneous variant is also of interest. Recall that in view of the classical Poincar\'{e} inequality on balls (e.g., \cite[Theorem~6.30]{AF}), the homogeneous case is only interesting when $R=\infty$. The \emph{homogeneous Sobolev space} $\dot{W}^{n,p}(\rd)$ is defined as the completion of compactly supported $\Cinf(\rd)$-functions with respect to the norm
\begin{equation*}
    \|f\|_{\dot{W}^{n,p}(\rd)} = \|\nabla^n f\|_{L^{p}(\rd)}.
\end{equation*}

Recall that the layer cake representation of the integral (e.g., \cite[Theorem~1.13]{LL}) shows that
\begin{equation}\label{E:layer_cake}
    \|h(r(x))\|_{L^p(B_R)} = (d\omega_d)^{\frac1{p}} \|h\|_{\Lpd(0,R)}
\end{equation}
for every measurable function $h$ on $(0, R)$, where $\omega_d = \pi^{d/2}/\Gamma(1 + d/2)$ denotes the volume of the unit ball in $\rd$. Moreover, when $h$ is an even or odd measurable function on $(-R,R)$, we clearly have
\begin{equation}\label{E:integral_symmetry}
    \|h\|_{\Lpd(-R,R)} = 2^{1/p} \|h\|_{\Lpd(0,R)}.
\end{equation}

Using this notation and Theorem~\ref{thm:main'} allows us to derive the following characterizations
of subspaces of radially symmetric functions of homogeneous and inhomogeneous Sobolev spaces.
\begin{theorem}\label{thm:radial_Sob_char}
    Let $p\in[1, \infty)$, $d\geq2$, and $n\in\N$. Let $R\in(0, \infty]$.
    
    A radial function $f(x) = g(r(x))$ on $B_R$ belongs to $W^{n,p}(B_R)$ if and only if its radial profile $g$ belongs to the closure of $\Cinf_{even}([-R,R])$-functions, which are compactly supported when $R=\infty$, with respect to the norm
    \begin{equation}\label{E:radial_Sob_char_norm_equiv_inhom}
        \sum_{k = 0}^n \sum_{j = 0}^{\lfloor k/2 \rfloor} \|(\D^j g)^{(k-2j)}\|_{\Lpd(-R,R)}.
    \end{equation}
    
Similarly, a radial function $f(x) = g(r(x))$ on $\rd$ belongs to $\dot{W}^{n,p}(\rd)$ if and only if its radial profile $g$ belongs to the closure of compactly supported $\Cinf_{even}(\R)$-functions with respect to the norm
    \begin{equation}\label{E:radial_Sob_char_norm_equiv_hom}
        \sum_{j = 0}^{\lfloor n/2 \rfloor} \|(\D^j g)^{(n-2j)}\|_{\Lpd(\R)}.
    \end{equation}
\end{theorem}
\begin{proof}
First, note that $\D g$ is an even function outside $0$ when $g$ is even. Hence, the derivatives of $\D^j g$ are even or odd functions outside $0$. Furthermore, by a density argument (recall also the discussion above \eqref{E:radial_laplacian}), we only need to show that $\|f\|_{W^{n,p}(B_R)}$ or $\|f\|_{\dot{W}^{n,p}(\rd)}$ and \eqref{E:radial_Sob_char_norm_equiv_inhom} or \eqref{E:radial_Sob_char_norm_equiv_hom}, respectively, are equivalent whenever $f(x)=g(r(x))$ is a smooth radial function.

To this end, we observe that \eqref{E:layer_cake} translates into $\|(S_nf)^{1/2}\|_{L^p(B_R)}=(d\omega_d)^{\frac1{p}}\|(T_ng)^{1/2}\|_{\Lpd(0,R)}$.
By Theorem~\ref{thm:main'} (and using that $\alpha^d_{n,j}$ are all positive)
\begin{equation*}
    \Big| (\D^l g)^{(k-2l)} \Big| \leq \sqrt{T_k g} \leq \sqrt{ \max_{j = 0, \dots, \lfloor k/2 \rfloor} \alpha^d_{k,j} } \cdot \sum_{j=0}^{\lfloor k/2\rfloor} \Big| (\D^j g)^{(k-2j)} \Big|
\end{equation*}
for all $k=0,\dots, n$ and $l=0,\dots, \lfloor k/2\rfloor$. Hence, 
\begin{equation*}
C_1 \cdot \sum_{k=0}^n \max_{l=0,\dots,\lfloor k/2\rfloor} \|(\D^l g)^{(k-2l)}\|_{\Lpd(0,R)} \leq  \|f\|_{W^{n,p}(B_R)}
\leq C_2 \cdot \sum_{k=0}^n \sum_{j=0}^{\lfloor k/2\rfloor} \|(\D^j g)^{(k-2j)}\|_{\Lpd(0,R)},
\end{equation*}
where
\begin{equation*}
C_1 = (d\omega_d)^{\frac1{p}} \qquad\text{and}\qquad  C_2
= (d\omega_d)^{\frac1{p}} \max_{\substack{j=0,\dots,\lfloor k/2 \rfloor \\ k=0,\dots,n}} \sqrt{\alpha^d_{k,j} }
= (d\omega_d)^{\frac1{p}} \sqrt{\alpha^d_{n,\lfloor n/2 \rfloor}}.
\end{equation*}
Similarly, we obtain
\begin{equation*}
C_1 \cdot \max_{j=0,\dots,\lfloor n/2\rfloor} \|(\D^j g)^{(n-2j)}\|_{\Lpd(0,\infty)} \leq  \|f\|_{\dot{W}^{n,p}(\rd)}
\leq C_2 \cdot \sum_{j=0}^{\lfloor n/2\rfloor} \|(\D^j g)^{(n-2j)}\|_{\Lpd(0,\infty)}.
\end{equation*}
The rest follows from \eqref{E:integral_symmetry}.
\end{proof}

\subsection{Comparison to known characterizations}

A different characterization was recently obtained in \cite{Ostermann}, with the inhomogeneous case being limited to $R$ finite. It can be formulated
as Theorem~\ref{thm:radial_Sob_char} but with \eqref{E:radial_Sob_char_norm_equiv_inhom} and \eqref{E:radial_Sob_char_norm_equiv_hom} replaced by
$\sum_{k = 0}^n \|t^k \D^k g(t)\|_{\Lpd(-R,R)}$ and $\|t^n \D^n g(t)\|_{\Lpd(\R)}$, respectively.

Although both our and that from \cite{Ostermann} characterizations involve iterations of the operator $\D$, neither is an obvious consequence of the other, and they provide different insights into the Sobolev norms of radial functions. We now show how Theorem~\ref{thm:radial_Sob_char} implies the latter. By a density argument and \eqref{E:integral_symmetry}, we only need to show that
\begin{equation}\label{E:radial_inhom_Sob_char_Ostermann}
    \sum_{k = 0}^n \sum_{j = 0}^{\lfloor k/2 \rfloor} \|(\D^j g)^{(k-2j)}\|_{\Lpd(0,R)} \approx \sum_{k = 0}^n \|r^k \D^k g(r)\|_{\Lpd(0,R)}
\end{equation}
and
\begin{equation}\label{E:radial_hom_Sob_char_Ostermann}
    \sum_{j = 0}^{\lfloor n/2 \rfloor} \|(\D^j g)^{(n-2j)}\|_{\Lpd(0, \infty)} \approx  \|r^n \D^n g(r)\|_{\Lpd(0,\infty)}
\end{equation}
for smooth radial functions $f(x) = g(r(x))$.
To this end, we need to establish three lemmas. Furthermore, it is worth pointing out that Theorem~\ref{thm:radial_Sob_char} is essentially a straightforward consequence of the pointwise identity obtained in Theorem~\ref{thm:main'} and does not involve much more than the lattice property of the Lebesgue norms. In particular, Theorem~\ref{thm:main'} easily implies similar characterizations in more general Sobolev-type spaces built on different function spaces than Lebesgue spaces. On the other hand, the characterizations 
obtained in \cite{Ostermann} involve Hardy-type inequalities, and their potential generalizations to more general Sobolev-type spaces are far less obvious. We will discuss this in more detail in Section~\ref{sec:subspace_generalization}.

\bigskip


{\bf Proof of the lower bounds in \eqref{E:radial_inhom_Sob_char_Ostermann} and \eqref{E:radial_hom_Sob_char_Ostermann}}

 \begin{lemma}\label{lem:RnDn_linear_combination}
    Let $n\geq1$ and let $g$ be a smooth function defined on an open interval not containing $0$. Then 
    \begin{equation}\label{eq:eqrDr}
        r^n \D^{n}g(r) = \sum_{j = 0}^{\lfloor n/2 \rfloor} c_{n,j} (\D^{j}g)^{(n - 2j)}(r),
    \end{equation}
    where
\[
 c_{n,0}=1,\quad c_{n,j}=(-1)^j \cdot \binom{n}{2j}\cdot (2j-1)!!, \quad 1\le j\le \lfloor n/2\rfloor.
\]
\end{lemma}
\begin{proof}
    The proof follows by induction. Indeed, for $n=1$ and $n=2$ one can verify \eqref{eq:eqrDr} directly. We now assume that \eqref{eq:eqrDr} is true for some $n-1$ and $n$.
    Differentiating \eqref{eq:eqrDr}, we obtain
\begin{equation}\label{eq:eqrDr1}
        nr^{n-1} \D^{n}g(r) +r^{n+1}\D^{n+1}g(r) = \sum_{j = 0}^{\lfloor n/2 \rfloor} c_{n,j} (\D^{j}g)^{(n - 2j +1)}(r).
    \end{equation}
Similarly, we apply \eqref{eq:eqrDr} with $n$ replaced by $n-1$ to $\D g$ instead of $g$ and obtain
\begin{equation}\label{eq:eqrDr2}
    r^{n-1} \D^n g(r) = \sum_{j = 0}^{\lfloor (n-1)/2 \rfloor} c_{n-1,j} (\D^{j+1}g)^{(n - 1 - 2j)}(r) = \sum_{j = 1}^{\lfloor (n+1)/2 \rfloor} c_{n-1,j-1} (\D^{j}g)^{(n + 1 - 2j)}(r).
\end{equation}
Now, we substitute \eqref{eq:eqrDr2} into \eqref{eq:eqrDr1} and obtain
\begin{align*}
    r^{n+1}\D^{n+1}g(r) = \sum_{j = 0}^{\lfloor n/2 \rfloor} c_{n,j} (\D^{j}g)^{(n + 1 - 2j)}(r) - n\cdot\sum_{j = 1}^{\lfloor (n+1)/2 \rfloor} c_{n-1,j-1} (\D^{j}g)^{(n + 1 - 2j)}(r).
\end{align*}
When $n$ is even, and so $n/2 = \lfloor n/2 \rfloor  = \lfloor (n+1)/2 \rfloor$, the desired validity of \eqref{eq:eqrDr} for $n+1$ follows from this upon observing that $c_{n,0} = c_{n+1, 0} = 1$ and
\begin{equation*}
    c_{n,j} - n\cdot c_{n-1, j-1} = c_{n+1, j} \quad \text{for every $j=1,\dots, \lfloor n/2 \rfloor$}.
\end{equation*}
Finally, when $n$ is odd, and so $(n+1)/2 = \lfloor (n+1)/2 \rfloor = \lfloor n/2 \rfloor + 1$, we need to combine these two facts with
\begin{equation*}
    c_{n+1, (n+1)/2} = -n\cdot c_{n-1, (n-1)/2}
\end{equation*}
to obtain the desired formula for $n+1$.
\end{proof}

Note that the preceding lemma implies that
\begin{equation*}
    \sum_{k = 0}^n |r^k \D^k g(r)| \leq C_3 \sum_{k = 0}^n \sum_{j = 0}^{\lfloor k/2 \rfloor} |(\D^{j}g)^{(k - 2j)}(r)|
\end{equation*}
and
\begin{equation*}
    |r^n \D^n g(r)| \leq C_3 \sum_{j = 0}^{\lfloor n/2 \rfloor} |(\D^{j}g)^{(n - 2j)}(r)|,
\end{equation*}
where
\begin{equation*}
    C_3 = \max_{\substack{j=0,\dots,\lfloor k/2 \rfloor \\ k=0,\dots,n}} |c_{k,j}|.
\end{equation*}
In particular, these pointwise estimates yield the $\gtrsim$ inequalities in \eqref{E:radial_inhom_Sob_char_Ostermann} and \eqref{E:radial_hom_Sob_char_Ostermann} for smooth radial functions $f(x) = g(r(x))$.

\bigskip

{\bf Proof of the upper bounds in \eqref{E:radial_inhom_Sob_char_Ostermann} and \eqref{E:radial_hom_Sob_char_Ostermann}}

\bigskip

To derive the reverse inequalities in \eqref{E:radial_inhom_Sob_char_Ostermann} and \eqref{E:radial_hom_Sob_char_Ostermann}, we start with the following lemma.
\begin{lemma}\label{lem:4}
Let $j,m\in\N_0$ and let $g$ be a smooth function defined on an open interval not containing $0$. Then 
    \begin{equation}\label{E:more_expl:1'}
    (\D^j g)^{(m)}(r) = \sum_{k = 0}^{\lfloor m/2 \rfloor} \beta_{m,k} r^{m - 2k} \D^{m + j - k}g(r),
\end{equation}
where the positive integers $\beta_{m,k}$ are defined as $\beta_{m,0} = 1$ and
\begin{equation*}
        \beta_{m,k} = \binom{m}{2k} \cdot (2k-1)!! \qquad \text{for $k=1,\dots, \lfloor m/2 \rfloor$}.
\end{equation*}
\end{lemma}
\begin{proof}
To verify the identity \eqref{E:more_expl:1'}, we can proceed by induction over $m$. The cases $m=0$ and $m=1$
are easily verified for all $j\in \N_0$. 
In the inductive step, one writes $(\D^j g)^{(m+1)} = ((\D^j g)^{(m)})'$ and uses the inductive hypothesis together with the product rule, which leads to two sums. After shifting the index up by one in one sum and using the fact that $(\D^{m + j - k}g)'(r) = r \D^{m + 1 + j - k}g(r)$ in the other, the rest follows by straightforwardly comparing the coefficients with those in \eqref{E:more_expl:1'} with $m+1$. To this end, one uses the identities $\beta_{m,0} = \beta_{m+1,0} = 1$,
\begin{equation*}
    (m-2k+2)\beta_{m,k-1}+\beta_{m,k}=\beta_{m+1,k}\quad \text{for $1\le k\le \lfloor (m-1)/2 \rfloor$},
\end{equation*}
and either
\begin{equation*}
    \beta_{m,m/2-1}\cdot 2+\beta_{m,m/2}=\beta_{m+1,m/2}\quad\text{if $m$ is even},
\end{equation*}
or
\begin{equation*}
    \beta_{m,(m-1)/2}=\beta_{m+1,(m+1)/2}\quad\text{if $m$ is odd}. \qedhere
\end{equation*}
\end{proof}

To prove the reversed inequality in \eqref{E:radial_hom_Sob_char_Ostermann}, we observe that it is enough to show that
\[
\|(\D^j g)^{(n-2j)}\|_{\Lpd(0,\infty)} \lesssim \|r^n \D^n g(r)\|_{\Lpd(0,\infty)}
\]
for $j=0,\dots,\lfloor n/2\rfloor$. Furthermore, by \eqref{E:more_expl:1'} applied with $m=n-2j$ this reduces to
\begin{equation}\label{eq:rDg}
\|r^{n-2\ell}(\D^{n-\ell} g)(r)\|_{\Lpd(0,\infty)} \lesssim \|r^n \D^n g(r)\|_{\Lpd(0,\infty)},\quad \ell=0,\dots,\lfloor n/2\rfloor.
\end{equation}
Finally, \eqref{eq:rDg} is obviously true for $\ell=0$ and it follows for other values of $\ell$ by iterating the following lemma.
\begin{lemma}
Let $p\in[1, \infty)$ and $\delta>-d/p$. Then there is a constant $C>0$ such that
    \begin{equation*}
          \|r^\delta g(r) \|_{\Lpd(0,\infty)} \leq C \|r^{\delta+2} \D g(r)\|_{\Lpd(0,\infty)}
    \end{equation*}
    for every smooth function $g$ in $(0, \infty)$ with bounded support.
\end{lemma}
\begin{proof}
Recall that a classical one-dimensional Hardy's inequality (cf.~\cite[Chapter~3, Lemma~3.9]{BS}) tells us that
\begin{equation}\label{eq:Hardy_general}
\int_0^\infty \Big(\int_r^\infty h(s) ds\Big)^p r^{\alpha - 1} \,dr \leq \Big(\frac{p}{\alpha}\Big)^p \int_0^\infty h(r)^p r^{p + \alpha - 1} \,dr,
\end{equation}
for all nonnegative measurable functions $h$ on $(0, \infty)$, $\alpha>0$, and $p\in[1,\infty)$.
Then
\begin{align*}
 \|r^\delta g(r)\|^p_{\Lpd(0,\infty)}&=\int_0^\infty r^{\delta p+d-1}|g(r)|^pdr
 \le \int_0^\infty r^{\delta p+d-1}\left(\int_r^\infty |g'(s)|ds\right)^pdr\\
 &\lesssim \int_0^\infty |g'(r)|^p r^{(\delta+1)p} r^{d-1}dr=\int_0^\infty |\D g(r)|^p r^{(\delta+2)p} r^{d-1}dr \\
 &=\|r^{\delta+2} \D g(r)\|^p_{\Lpd(0,\infty)}. \qedhere
\end{align*}
\end{proof}

To prove the reverse inequality in \eqref{E:radial_inhom_Sob_char_Ostermann}, we proceed similarly, but we have to deal also with the boundary terms.
We rely on the following lemma.
\begin{lemma}\label{lem:h}
Let $R\in(0, \infty)$, $p\in[1, \infty)$, and $\delta>-d/p$. Then there is a constant $C>0$ such that
\begin{equation}\label{lem:h:eq}
    \|r^\delta g(r)\|_{\Lpd(0,R)} \leq C\Big( \|r^{\delta+2} \D g(r)\|_{\Lpd(0,R)} + \|r^{\delta+1} g(r)\|_{\Lpd(0,R)} \Big)
\end{equation}

for every smooth function $g$ on $(0, R]$.
\end{lemma}
\begin{proof}
We clearly have
\begin{equation}\label{eq:h:1}
    \|r^\delta g(r)\|_{\Lpd(0,R)} \leq \big\|r^\delta\big( g(r) - g(R) \big) \big\|_{\Lpd(0,R)} + |g(R)| \cdot \|r^\delta\|_{\Lpd(0,R)}.
\end{equation}
Moreover, we clearly have $\|r^{\delta} \|_{\Lpd(0,R)} < \infty$ thanks to the fact that $\delta>-d/p$. Furthermore, writing
    \begin{equation*}
        |g(r) - g(R)| \leq \int_r^R s |\D g(s)| \,ds
    \end{equation*}
    and using \eqref{eq:Hardy_general} with $h(s)=s |\D g(s)|\chi_{(0,R)}(s)$ and $\alpha = d + \delta p > 0$, we obtain
    \begin{equation}\label{eq:h:2}
        \big\|r^\delta\big( g(r) - g(R) \big) \big\|_{\Lpd(0,R)} \leq \frac{p}{d + \delta p} \|r^{\delta + 2} \D g(r)\|_{\Lpd(0,R)}.
    \end{equation}
    Now, the elementary mean value theorem for integrals yields the existence of $r_0\in[\frac{R}{2},R]$ such that
\begin{equation*}
    g(r_0) = \frac{2}{R} \int_{R/2}^R g(r) \, dr.
\end{equation*}
Using this, note that
\begin{align*}
    |g(R)| &= \Bigg| \int_{r_0}^R r \D g(r) \,dr + \frac{2}{R} \int_{R/2}^R g(r) \, dr \Bigg| \\
    &\leq \int_{R/2}^R r^{\delta + 2} |\D g(r)| r^{-d - \delta} r^{d-1}\,dr + \frac{2}{R} \int_{R/2}^R r^{\delta+1}|g(r)| r^{-d - \delta} r^{d-1}\, dr.
\end{align*}
Furthermore, using the H\"older inequality, we obtain
\begin{equation}\label{eq:h:3}
    |g(R)| \leq \|r^{-d - \delta}\|_{L^{p'}_{d-1}(R/2,R)}\cdot \Big(\|r^{\delta + 2} \D g(r)\|_{\Lpd(0,R)} + \frac{2}{R} \|r^{\delta + 1} g(r)\|_{\Lpd(0,R)} \Big).
\end{equation}
Hence, combining \eqref{eq:h:1}--\eqref{eq:h:3}, we obtain \eqref{lem:h:eq}.
\end{proof}

By Lemma~\ref{lem:4}, to establish the $\lesssim$ inequality in \eqref{E:radial_inhom_Sob_char_Ostermann}, it is enough to show that
\begin{equation}\label{eq:upper1}
\|r^{k-2\ell}\D^{k-\ell}g\|_{\Lpd(0,R)}\lesssim \sum_{m=0}^n \|r^m\D^mg\|_{\Lpd(0,R)}
\end{equation}
holds for every $n\ge 1$, every $k=0,\dots,n$ and every $\ell=0,\dots,\lfloor k/2\rfloor$. Clearly, \eqref{eq:upper1}
is true if $\ell=0$ and $k$ is arbitrary. The remaining cases follow by induction. 
We assume that \eqref{eq:upper1} holds for $k-1$ and all admissible $\ell$'s and for some pair $(k,\ell-1)$. Then
we obtain by Lemma~\ref{lem:h}
\begin{align*}
\|r^{k-2\ell}\D^{k-\ell}g\|_{\Lpd(0,R)}\lesssim
\|r^{k-2(\ell-1)}\D^{k-(\ell-1)}g\|_{\Lpd(0,R)}+
\|r^{k-1-2(\ell-1)}\D^{(k-1)-(\ell-1)}g\|_{\Lpd(0,R)}.
\end{align*}
Finally, both these terms are bounded by the right-hand side of \eqref{eq:upper1},
as we assume that \eqref{eq:upper1} is true for both the pairs $(k,\ell-1)$ and $(k-1,\ell-1)$.
We conclude that \eqref{eq:upper1} holds also for the pair $(k,\ell).$


\bigskip

The pointwise nature of Theorem~\ref{thm:main'} enables us to obtain a certain variant of Theorem~\ref{thm:radial_Sob_char}
also for $p=\infty$. For that sake, we denote by $\dot{W}^{n,\infty}_{*}(\rd)$ the completion of all smooth compactly supported functions on $\rd$ with
respect to the norm $\|f\|_{\dot{W}_*^{n,\infty}(\rd)} = \|\nabla^n f\|_{L^{\infty}(\rd)}$. 
Again, we conclude that a radial function $f(x)=g(r(x))$
belongs to $\dot{W}_*^{n,\infty}(\rd)$ if, and only if, its radial profile $g$ belongs to the closure of the space of smooth compactly supported even
functions on $\R$ with respect to the norm $\sum_{j=0}^{\lfloor n/2\rfloor} \|(\D^j g)^{(n-2j)}\|_{L^\infty(\R)}$. 
In an analogous way, one can handle also the inhomogeneous spaces $W_*^{n,\infty}(B_R)$.

\section{Further remarks and open problems}\label{sec:4}

We add several comments and remarks to our main results and sketch possible extensions and open problems.

\subsection{Radial symmetry of \texorpdfstring{$S_n f$}{S\_n f}}

It is not immediately clear from the definition of $S_nf$, cf.~\eqref{eq:defS_n},
that $S_n f$ is radially symmetric if $f$ is a radially symmetric function.
In our approach, this comes as a byproduct of Theorem~\ref{thm:main'}.
On the other hand, it can also be shown directly, using only elementary facts about the Frobenius norm and tensors.
We sketch the details of a simple self-containing argument, which avoids extensive references to the literature on tensors.

\medskip

We denote by
\[
A=(A_{i_1,\dots,i_n})_{i_1,\dots,i_n=1}^d
\]
a tensor of $n$-th order and by
\[
A(u_1,\dots,u_n)=\sum_{i_1,\dots,i_n=1}^d A_{i_1,\dots,i_n}u_{1,i_1}u_{2,i_2}\dots u_{n,i_n}
\]
its action on an $n$-tuple of vectors $(u_1,\dots,u_n)$. The Frobenius norm of $A$ is defined as
\[
\|A\|_F^2=\sum_{i_1,\dots,i_n=1}^d [A_{i_1,\dots,i_n}]^2=\sum_{i_1,\dots,i_n=1}^d [A(\cVec{i_1},\dots,\cVec{i_n})]^2,
\]
where $\cVec{1},\dots,\cVec{d}$ are the canonical basis vectors of $\rd$.

\medskip

We shall rely on the following two facts, which are easily verified by a direct calculation.
If $n=2$, then the first one reflects the very well-known fact that the Frobenius norm of a matrix
is unitarily invariant, see \cite[Chapter 5.6]{HJ}. The second one follows just by iterated chain rule.
\begin{lemma}\label{lem:sym}
(i) Let $\varphi_1,\dots,\varphi_d$ be any orthonormal basis of $\R^d$. Then
\[
\|A\|_F^2=\sum_{i_1,\dots,i_n=1}^d [A(\varphi_{i_1},\dots,\varphi_{i_n})]^2.
\]
(ii) Let $f$ be a smooth function of $d$ variables, let $U\in\R^{d\times d}$ be an orthonormal matrix and let $g(x)=f(Ux)$. Then
\[
\nabla^ng(x)(\cVec{i_1},\dots,\cVec{i_n})=\nabla^nf(Ux)(u_{i_1},\dots,u_{i_n}),
\]
where $u_1,\dots,u_d$ are the columns of $U$.
\end{lemma}

Let now $f$ be radially symmetric and let $U$ be an orthonormal matrix. Then $f(x)=f(Ux)$ and Lemma~\ref{lem:sym} gives 
\[
\|\nabla^n f(x)\|_F^2 = \sum_{i_1,\dots,i_n=1}^d [\nabla^nf(x)(\cVec{i_1},\dots,\cVec{i_n})]^2= \sum_{i_1,\dots,i_n=1}^d [\nabla^nf(Ux)(u_{i_1},\dots,u_{i_n})]^2=\|\nabla^n f(Ux)\|_F^2.
\]
Therefore, $S_nf(x)=S_nf(Ux)$ and $S_nf(x)$ is also radially symmetric.

\subsection{Explicit formula for \texorpdfstring{$T_n$}{T\_n}}\label{sec:worked_out_formula} 
There are several ways how to rewrite the operator $T_n$ defined in \eqref{eq:defT'}.
Especially, one might prefer to avoid the mixture of the (usual) derivatives and the higher iterations of the operator $\D$.
We will show that without much additional effort, one can work with only the iterations of $\D$,
or with higher derivatives of $g$, respectively.

\medskip

Indeed, first we can apply Lemma~\ref{lem:4}
to reformulate $(\D^j g)^{(n-2j)}$, which transforms \eqref{eq:defT'} into
\begin{equation}\label{E:more_expl:2}
    (T_ng)(r) = \sum_{j=0}^{\lfloor n/2\rfloor} \alpha^d_{n,j} \left[\sum_{k = 0}^{\lfloor (n-2j)/2 \rfloor} \beta_{n-2j,k} r^{n-2j- 2k} \D^{n - j - k}g(r)\right]^2.
\end{equation}
Furthermore, since iterating the operator $\D$ essentially means to apply the product rule and it is easy to differentiate power functions, it is not hard to work out that, for $m\in\N$,
\begin{equation*}
    \D^{m}g(r)=\sum_{l=0}^{m - 1}c_{m+l-1,l} \frac{g^{(m - l)}(r)}{r^{m + l}},
\end{equation*}
where the coefficients $c_{m+l-1,l}$ are those from Lemma~\ref{lem:RnDn_linear_combination}. Finally, plugging it into \eqref{E:more_expl:2}, one obtains
\begin{equation*}
    (T_ng)(r) = \sum_{j=0}^{\lfloor n/2\rfloor} \alpha^d_{n,j} \left[\sum_{k = 0}^{\lfloor (n-2j)/2 \rfloor} \beta_{n-2j,k} \sum_{l=0}^{n - j - k - 1} c_{n-j-k+l-1,l} \frac{g^{(n - j - k - l)}(r)}{r^{j+k+l}} \right]^2.
\end{equation*}

\subsection{Modified radial profile \texorpdfstring{$f(x) = h(r(x)^2)$}{f(x) = h(r(x)\textasciicircum{}2)}}
Instead of the radial profile $f(x) = g(r(x))$, it is sometimes more convenient to work with the modified
radial profile $f(x) = h(r(x)^2)$, where $h(r^2) = g(r)$. One can verify that, for $j,m\in\N_0$,
\begin{equation*}
    (\D^j g)^{(m)}(r) = 2^{m+j} (\widetilde{D}^{m}h^{(j)})(r^2),
\end{equation*}
where $\widetilde{D}h(r) = \sqrt{r} \cdot h'(r)$.
Plugging this into \eqref{eq:defT'}, we obtain 
\begin{align*}
S_n f(x)= (T_ng)(r)=
\sum_{j=0}^{\lfloor n/2\rfloor} \alpha^d_{n,j} \left[(\D^j g)^{(n-2j)}(r)\right]^2
= \sum_{j=0}^{\lfloor n/2\rfloor} \alpha^d_{n,j} \left[ 2^{n-j} \big( \widetilde{D}^{n-2j}h^{(j)} \big)(r^2) \right]^2
=:(\widetilde{T}_n h)(r^2).
\end{align*}

\subsection{Subspace of radial functions in general Sobolev-type spaces}\label{sec:subspace_generalization}

In Section~\ref{sec:subspace}, we showed how one can use Theorem~\ref{thm:main'} to obtain a characterization of the subspace of radial functions in the usual Sobolev spaces $W^{n,p}(B_R)$ or $\dot{W}^{n,p}(\rd)$ (recall Theorem~\ref{thm:radial_Sob_char}). Importantly, since the formula \eqref{eq:defT'} is pointwise, it is also easy to use Theorem~\ref{thm:main'} to characterize the subspace in more general Sobolev-type spaces. Note that this is in contrast with the characterizations and approaches, for example, from \cites{GdF-dS-M, Ostermann}, which used various Hardy-type inequalities.
In fact, we may replace Lebesgue spaces $L^p$ in the definition of the Sobolev space with quite general Banach lattices of functions. We present it here only for inhomogeneous Sobolev spaces over $\rd$ and leave the remaining cases, which require only simple technical modifications, to the interested reader.

Let $X$ be a Banach lattice of functions on $\rd$, that is, a Banach space of (equivalence classes of) measurable functions on $\rd$ whose norm satisfies the lattice property:
\begin{equation*}
    \text{if $f\in X$ and $|g|\leq |f|$ a.e., then $g\in X$ and $\|g\|_X \leq \|f\|_X$}.
\end{equation*}
Assume that $X$ contains bounded functions supported on compact sets and is continuously embedded in the space of locally integrable functions, that is,
\begin{equation*}
    L^\infty_c(\rd) \subseteq X \hookrightarrow L^1_{loc}(\rd).
\end{equation*}
Recall that $X \hookrightarrow L^1_{loc}(\rd)$ means that for every compact set $K\subseteq\rd$, there is a constant $C_K$ such that $\|f\|_{L^1(K)}\leq C_K \|f\|_X$ for every $f\in X$. As an important example, $X$ can be any Banach function space (see~\cite{BS}), which includes Orlicz spaces, Lorentz spaces, variable exponent Lebesgue spaces, or some weighted Lebesgue spaces. We then define the $n$th order inhomogeneous Sobolev space built upon $X$, and denote it $W_0^n X(\rd)$, as the completion of smooth compactly supported functions on $\rd$ with respect to the norm
\begin{equation*}
    \|f\|_{W_0^n X(\rd)} = \sum_{k = 0}^n \|\nabla^k f\|_{X}.
\end{equation*}
As before, for brevity we write $\|\nabla^k f\|_{X}$ instead of $\| \|\nabla^k f\|_F\|_{X}$.

Now, using Theorem~\ref{thm:main'} and the lattice property of $X$, we straightforwardly obtain
\begin{equation}\label{E:subspace_generalization:1}
    \|f\|_{W_0^n X(\rd)} \approx \sum_{k=0}^n \sum_{j=0}^{\lfloor k/2\rfloor} \|(\D^j g)^{(k-2j)}(r(x))\|_{X}
\end{equation}
for every smooth radial function $f(x) = g(r(x))$ with compact support.

Moreover, assume that $X$ is representable over $(0, \infty)$ in the sense that there is a Banach lattice of functions on $(0,\infty)$, denoted $\widebar{X}$, such that
\begin{equation}\label{E:subspace_generalization:5}
    \|h(r(x))\|_{X} \approx \|h(r^\frac1{d})\|_{\widebar{X}} \quad \text{for every measurable $h$ on $(0, \infty)$}.
\end{equation}
For example, this is the case for all rearrangement-invariant function spaces in the sense of \cite{BS} (the interested reader is referred to \cite{CPS} and references therein for more information about Sobolev spaces built upon these spaces).
Then, using \eqref{E:subspace_generalization:1}, \eqref{E:subspace_generalization:5}, and a density argument, we have the following characterization in the spirit of Theorem~\ref{thm:radial_Sob_char}. A radial function $f(x) = g(r(x))$ belongs to $W_0^n X(\rd)$ if and only if its radial profile $g$ belongs to the closure of compactly supported $\Cinf_{even}(\R)$-functions with respect to the norm
\begin{equation*}
    \sum_{k=0}^n \sum_{j=0}^{\lfloor k/2\rfloor} \|(\D^j g)^{(k-2j)}(r^\frac1{d})\|_{\widebar{X}}.
\end{equation*}
Note that \eqref{E:layer_cake}, after the obvious change of variables, is a special case of \eqref{E:subspace_generalization:5}. However, for technical reasons (cf.~\cite[pp.~31-32]{Baernstein}), it is better to write \eqref{E:subspace_generalization:5} in the form with $r^\frac1{d}$ instead of $r$ in this general setting. Finally, we note that it would be possible to generalize this even to the quasi-Banach setting, but we do not dive into it here.

\subsection{Open problems}

There are several questions closely related to Theorem~\ref{thm:main'}, which were not addressed in our work and which we believe to be worth further investigation.
We state them here in the hope that they might attract the attention of other researchers.
\begin{enumerate}
    \item As mentioned in the introduction, the elegant (and very short) paper \cite{Lyons-Zumbrun} provided
    the striking formula \eqref{E:Lyons-Zumbrun} for an arbitrary partial derivative of a radial function.
    This relation was recently used in \cite{Ostermann} to characterize the subspaces of Sobolev spaces consisting of radial functions.
    We leave it as an open problem whether one could also exploit \cite{Lyons-Zumbrun} to provide an alternative proof of Theorem~\ref{thm:main'}.
    \item As mentioned already in the Introduction, for $n=2$, \eqref{eq:defT'} reduces to the well-known relation \eqref{eq:Sn2}, which has a straightforward connection to the spectrum
    of the Hessian matrix $Hf(x)$. It would be interesting to know if there is an analogy of this connection also for higher $n\ge 3$.
    This could lead to better understanding of the role of the coefficients $\alpha^d_{n,j}$ in \eqref{eq:defT'}.
    Due to the well-known difficulties of spectral theory for tensors \cite{Tensors}, we do not dive into this question in this work.
    \item Our approach is tailored to radially symmetric functions defined on Euclidean spaces. It would be of course interesting to know
    if it can also be adapted to other types of symmetries. Furthermore, the role of symmetry is also classically studied in non-Euclidean geometries (for example, see~\cite{Hebey, MR4277332}). An important example is the hyperbolic geometry. For studying radial symmetry in the hyperbolic space, the Poincar\'e ball model, that is, the unit ball in $\rd$ equipped with a suitable Riemannian metric, is particularly useful because its conformal factor is radially symmetric. Recently, a characterization of the subspace of radial functions in Sobolev spaces on hyperbolic balls with finite radius in the Poincar\'e ball model was provided in \cite{DLP:25}. However, this characterization is in the spirit of its Euclidean counterpart from \cite{GdF-dS-M}, and as such entails the same restriction on the parameters involved. It would be of interest to find an elegant characterization without such restriction.
\end{enumerate}

\bigskip

{\bf Acknowledgment:} The work of both authors has been supported by the grant 23-04720S of the Czech Science Foundation.


\begin{biblist}

\bib{AF}{book}{
   author={Adams, Robert A.},
   author={Fournier, John J. F.},
   title={Sobolev spaces},
   series={Pure and Applied Mathematics (Amsterdam)},
   volume={140},
   edition={2},
   publisher={Elsevier/Academic Press, Amsterdam},
   date={2003},
   pages={xiv+305},
   isbn={0-12-044143-8},
   review={\MR{2424078}},
}

\bib{AP}{article}{
   author={Angelsberg, Gilles},
   author={Pumberger, David},
   title={A regularity result for polyharmonic maps with higher
   integrability},
   journal={Ann. Global Anal. Geom.},
   volume={35},
   date={2009},
   number={1},
   pages={63--81},
   issn={0232-704X},
   review={\MR{2480664}},
   doi={10.1007/s10455-008-9122-z},
}

\bib{Baernstein}{book}{
   author={Baernstein, Albert, II},
   title={Symmetrization in analysis},
   series={New Mathematical Monographs},
   volume={36},
   note={With David Drasin and Richard S. Laugesen;
   With a foreword by Walter Hayman},
   publisher={Cambridge University Press, Cambridge},
   date={2019},
   pages={xviii+473},
   isbn={978-0-521-83047-8},
   review={\MR{3929712}},
   doi={10.1017/9781139020244},
}

\bib{BS}{book}{
   author={Bennett, Colin},
   author={Sharpley, Robert},
   title={Interpolation of operators},
   series={Pure and Applied Mathematics},
   volume={129},
   publisher={Academic Press, Inc., Boston, MA},
   date={1988},
   pages={xiv+469},
   isbn={0-12-088730-4},
   review={\MR{0928802}},
}

\bib{B-L1}{article}{
   author={Berestycki, H.},
   author={Lions, P.-L.},
   title={Existence of a ground state in nonlinear equations of the
   Klein-Gordon type},
   conference={
      title={Variational inequalities and complementarity problems},
      address={Proc. Internat. School, Erice},
      date={1978},
   },
   book={
      series={Wiley-Intersci. Publ.},
      publisher={Wiley, Chichester},
   },
   isbn={0-471-27610-3},
   date={1980},
   pages={35--51},
   review={\MR{0578738}},
}

\bib{B-L2}{article}{
   author={Berestycki, H.},
   author={Lions, P.-L.},
   title={Nonlinear scalar field equations. I. Existence of a ground state},
   journal={Arch. Rational Mech. Anal.},
   volume={82},
   date={1983},
   number={4},
   pages={313--345},
   issn={0003-9527},
   review={\MR{0695535}},
   doi={10.1007/BF00250555},
}

\bib{B-L3}{article}{
   author={Berestycki, H.},
   author={Lions, P.-L.},
   title={Nonlinear scalar field equations. II. Existence of infinitely many
   solutions},
   journal={Arch. Rational Mech. Anal.},
   volume={82},
   date={1983},
   number={4},
   pages={347--375},
   issn={0003-9527},
   review={\MR{0695536}},
   doi={10.1007/BF00250556},
}

\bib{CPS}{article}{
   author={Cianchi, Andrea},
   author={Pick, Lubo\v s},
   author={Slav\'ikov\'a, Lenka},
   title={Higher-order Sobolev embeddings and isoperimetric inequalities},
   journal={Adv. Math.},
   volume={273},
   date={2015},
   pages={568--650},
   issn={0001-8708},
   review={\MR{3311772}},
   doi={10.1016/j.aim.2014.12.027},
}

\bib{C-G-M}{article}{
   author={Coleman, S.},
   author={Glaser, V.},
   author={Martin, A.},
   title={Action minima among solutions to a class of Euclidean scalar field
   equations},
   journal={Comm. Math. Phys.},
   volume={58},
   date={1978},
   number={2},
   pages={211--221},
   issn={0010-3616},
   review={\MR{0468913}},
}

\bib{DLP:25}{article}{
   author={Do \'O, Jo\~ao Marcos},
   author={Lu, Guozhen},
   author={Ponciano, Raon\'i},
   title={Sharp Sobolev and Adams-Trudinger-Moser inequalities for symmetric
   functions without boundary conditions on hyperbolic spaces},
   journal={Calc. Var. Partial Differential Equations},
   volume={64},
   date={2025},
   number={8},
   pages={Paper No. 255, 42},
   issn={0944-2669},
   review={\MR{4959099}},
   doi={10.1007/s00526-025-03131-1},
}

\bib{Duchon}{article}{
   author={Duchon, Jean},
   title={Splines minimizing rotation-invariant semi-norms in Sobolev
   spaces},
   conference={
      title={Constructive theory of functions of several variables},
      address={Proc. Conf., Math. Res. Inst., Oberwolfach},
      date={1976},
   },
   book={
      series={Lecture Notes in Math.},
      volume={Vol. 571},
      publisher={Springer, Berlin-New York},
   },
   date={1977},
   pages={85--100},
   review={\MR{0493110}},
}

\bib{MR4277332}{article}{
   author={Farkas, Csaba},
   author={Krist\'aly, Alexandru},
   author={Mester, \'Agnes},
   title={Compact Sobolev embeddings on non-compact manifolds via orbit
   expansions of isometry groups},
   journal={Calc. Var. Partial Differential Equations},
   volume={60},
   date={2021},
   number={4},
   pages={Paper No. 128, 31},
   issn={0944-2669},
   review={\MR{4277332}},
   doi={10.1007/s00526-021-01997-5},
}

\bib{GS}{article}{
   author={Gastel, Andreas},
   author={Scheven, Christoph},
   title={Regularity of polyharmonic maps in the critical dimension},
   journal={Comm. Anal. Geom.},
   volume={17},
   date={2009},
   number={2},
   pages={185--226},
   issn={1019-8385},
   review={\MR{2520907}},
   doi={10.4310/CAG.2009.v17.n2.a2},
}

\bib{GdF-dS-M}{article}{
   author={Guedes de Figueiredo, Djairo},
   author={dos Santos, Ederson Moreira},
   author={Miyagaki, Ol\'impio Hiroshi},
   title={Sobolev spaces of symmetric functions and applications},
   journal={J. Funct. Anal.},
   volume={261},
   date={2011},
   number={12},
   pages={3735--3770},
   issn={0022-1236},
   review={\MR{2838041}},
   doi={10.1016/j.jfa.2011.08.016},
}

\bib{Hebey}{book}{
   author={Hebey, Emmanuel},
   title={Nonlinear analysis on manifolds: Sobolev spaces and inequalities},
   series={Courant Lecture Notes in Mathematics},
   volume={5},
   publisher={New York University, Courant Institute of Mathematical
   Sciences, New York; American Mathematical Society, Providence, RI},
   date={1999},
   pages={x+309},
   isbn={0-9658703-4-0},
   isbn={0-8218-2700-6},
   review={\MR{1688256}},
}

\bib{Tensors}{article}{
   author={Hillar, Christopher J.},
   author={Lim, Lek-Heng},
   title={Most tensor problems are NP-hard},
   journal={J. ACM},
   volume={60},
   date={2013},
   number={6},
   pages={Art. 45, 39},
   issn={0004-5411},
   review={\MR{3144915}},
   doi={10.1145/2512329},
}

\bib{HJ}{book}{
   author={Horn, Roger A.},
   author={Johnson, Charles R.},
   title={Matrix analysis},
   publisher={Cambridge University Press, Cambridge},
   date={1985},
   pages={xiii+561},
   isbn={0-521-30586-1},
   review={\MR{0832183}},
   doi={10.1017/CBO9780511810817},
}

\bib{Johnson:02}{article}{
   author={Johnson, Warren P.},
   title={The curious history of Fa\`a{} di Bruno's formula},
   journal={Amer. Math. Monthly},
   volume={109},
   date={2002},
   number={3},
   pages={217--234},
   issn={0002-9890},
   review={\MR{1903577}},
   doi={10.2307/2695352},
}

\bib{LL}{book}{
   author={Lieb, Elliott H.},
   author={Loss, Michael},
   title={Analysis},
   series={Graduate Studies in Mathematics},
   volume={14},
   edition={2},
   publisher={American Mathematical Society, Providence, RI},
   date={2001},
   pages={xxii+346},
   isbn={0-8218-2783-9},
   review={\MR{1817225}},
   doi={10.1090/gsm/014},
}

\bib{Lions-com}{article}{
   author={Lions, P.-L.},
   title={Sym\'etrie et compacit\'e{} dans les espaces de Sobolev},
   language={French, with English summary},
   journal={J. Functional Analysis},
   volume={49},
   date={1982},
   number={3},
   pages={315--334},
   issn={0022-1236},
   review={\MR{0683027}},
   doi={10.1016/0022-1236(82)90072-6},
}

\bib{Lions-vanA}{article}{
   author={Lions, P.-L.},
   title={The concentration-compactness principle in the calculus of
   variations. The locally compact case. I},
   language={English, with French summary},
   journal={Ann. Inst. H. Poincar\'e{} Anal. Non Lin\'eaire},
   volume={1},
   date={1984},
   number={2},
   pages={109--145},
   issn={0294-1449},
   review={\MR{0778970}},
}

\bib{Lions-vanB}{article}{
   author={Lions, P.-L.},
   title={The concentration-compactness principle in the calculus of
   variations. The locally compact case. II},
   language={English, with French summary},
   journal={Ann. Inst. H. Poincar\'e{} Anal. Non Lin\'eaire},
   volume={1},
   date={1984},
   number={4},
   pages={223--283},
   issn={0294-1449},
   review={\MR{0778974}},
}

\bib{Lions-conA}{article}{
   author={Lions, P.-L.},
   title={The concentration-compactness principle in the calculus of
   variations. The limit case. I},
   journal={Rev. Mat. Iberoamericana},
   volume={1},
   date={1985},
   number={1},
   pages={145--201},
   issn={0213-2230},
   review={\MR{0834360}},
   doi={10.4171/RMI/6},
}

\bib{Lions-conB}{article}{
   author={Lions, P.-L.},
   title={The concentration-compactness principle in the calculus of
   variations. The limit case. II},
   journal={Rev. Mat. Iberoamericana},
   volume={1},
   date={1985},
   number={2},
   pages={45--121},
   issn={0213-2230},
   review={\MR{0850686}},
   doi={10.4171/RMI/12},
}

\bib{Lyons-Zumbrun}{article}{
   author={Lyons, Russell},
   author={Zumbrun, Kevin},
   title={Homogeneous partial derivatives of radial functions},
   journal={Proc. Amer. Math. Soc.},
   volume={121},
   date={1994},
   number={1},
   pages={315--316},
   issn={0002-9939},
   review={\MR{1227524}},
   doi={10.2307/2160399},
}

\bib{Ostermann}{article}{
   author={Ostermann, Matthias},
   title={A characterization of the subspace of radially symmetric functions in Sobolev spaces},
   journal={Commun. Contemp. Math.},
   volume={27},
   date={2025},
   number={3},
   pages={Paper No. 2450018, 15},
   issn={0219-1997},
   review={\MR{4844376}},
   doi={10.1142/S0219199724500184},
}

\bib{SSV1}{article}{
   author={Sickel, Winfried},
   author={Skrzypczak, Leszek},
   author={Vybiral, Jan},
   title={On the interplay of regularity and decay in case of radial
   functions I. Inhomogeneous spaces},
   journal={Commun. Contemp. Math.},
   volume={14},
   date={2012},
   number={1},
   pages={1250005, 60},
   issn={0219-1997},
   review={\MR{2902295}},
   doi={10.1142/S0219199712500058},
}

\bib{MR3330617}{article}{
   author={Sickel, Winfried},
   author={Skrzypczak, Leszek},
   author={Vyb\'iral, Jan},
   title={The characterization of radial subspaces of Besov and
   Lizorkin-Triebel spaces by differences},
   conference={
      title={Function spaces X},
   },
   book={
      series={Banach Center Publ.},
      volume={102},
      publisher={Polish Acad. Sci. Inst. Math., Warsaw},
   },
   isbn={978-83-86806-25-6},
   date={2014},
   pages={197--214},
   review={\MR{3330617}},
   doi={10.4064/bc102-0-14},
}

\bib{Win1}{article}{
   author={Sickel, Winfried},
   author={Skrzypczak, Leszek},
   title={On the interplay of regularity and decay in case of radial
   functions II. Homogeneous spaces},
   journal={J. Fourier Anal. Appl.},
   volume={18},
   date={2012},
   number={3},
   pages={548--582},
   issn={1069-5869},
   review={\MR{2921084}},
   doi={10.1007/s00041-011-9205-2},
}

\bib{Win2}{article}{
   author={Sickel, Winfried},
   author={Skrzypczak, Leszek},
   title={Radial subspaces of Besov and Lizorkin-Triebel classes: extended
   Strauss lemma and compactness of embeddings},
   journal={J. Fourier Anal. Appl.},
   volume={6},
   date={2000},
   number={6},
   pages={639--662},
   issn={1069-5869},
   review={\MR{1790248}},
   doi={10.1007/BF02510700},
}

\bib{Stein}{book}{
   author={Stein, Elias M.},
   title={Singular integrals and differentiability properties of functions},
   series={Princeton Mathematical Series},
   volume={No. 30},
   publisher={Princeton University Press, Princeton, NJ},
   date={1970},
   pages={xiv+290},
   review={\MR{0290095}},
}

\bib{MR0454365}{article}{
   author={Strauss, Walter A.},
   title={Existence of solitary waves in higher dimensions},
   journal={Comm. Math. Phys.},
   volume={55},
   date={1977},
   number={2},
   pages={149--162},
   issn={0010-3616},
   review={\MR{0454365}},
}

\bib{Talenti}{article}{
   author={Talenti, Giorgio},
   title={Best constant in Sobolev inequality},
   journal={Ann. Mat. Pura Appl. (4)},
   volume={110},
   date={1976},
   pages={353--372},
   issn={0003-4622},
   review={\MR{0463908}},
   doi={10.1007/BF02418013},
}

\bib{W}{book}{
   author={Wahba, Grace},
   title={Spline models for observational data},
   series={CBMS-NSF Regional Conference Series in Applied Mathematics},
   volume={59},
   publisher={Society for Industrial and Applied Mathematics (SIAM),
   Philadelphia, PA},
   date={1990},
   pages={xii+169},
   isbn={0-89871-244-0},
   review={\MR{1045442}},
   doi={10.1137/1.9781611970128},
}

\bib{Whitney:43}{article}{
   author={Whitney, Hassler},
   title={Differentiable even functions},
   journal={Duke Math. J.},
   volume={10},
   date={1943},
   pages={159--160},
   issn={0012-7094},
   review={\MR{0007783}},
}

\end{biblist}
\end{document}